\documentclass[onefignum,onetabnum]{siamart220329}

\usepackage{graphicx}
\usepackage{amsmath}
\usepackage{amsfonts}
\usepackage{enumitem}
\usepackage[title]{appendix}

\usepackage{booktabs}

\usepackage{cleveref}

\crefname{equation}{Eq.}{Eqs.}
\Crefname{equation}{Equation}{Equations}
\crefname{figure}{Fig.}{Figs.}
\Crefname{figure}{Figure}{Figures}

\newsiamremark{remark}{Remark}
\newsiamremark{hypothesis}{Hypothesis}
\crefname{hypothesis}{Hypothesis}{Hypotheses}
\newsiamthm{claim}{Claim}

\DeclareMathOperator*{\argmin}{arg\,min}

\DeclareMathOperator{\prox}{prox}

\newcommand{\qqbox}[1]{\qquad\hbox{#1}\qquad}

\title{Fast convex optimization via inertial systems with asymptotically vanishing viscosity and Hessian-driven damping\thanks{Submitted to the editors DATE.
\funding{This work was partially funded by the China Scholarship Council~202208520010, and also benefited from the support of the FMJH Program Gaspard Monge for optimization and operations research and their interactions with data science.}}}

\author{Zepeng Wang\thanks{Bernoulli Institute for Mathematics, Computer Science and Artificial Intelligence, University of Groningen, Groningen, The Netherlands 
  (\email{zepeng.wang@rug.nl}).}
\and Juan Peypouquet\thanks{Bernoulli Institute for Mathematics, Computer Science and Artificial Intelligence, University of Groningen, Groningen, The Netherlands 
  (\email{j.g.peypouquet@rug.nl}).}
}

\begin{document}
\maketitle

% REQUIRED
\begin{abstract}
We study the convergence rate of a family of inertial algorithms, which can be obtained by discretization of an inertial system combining asymptotic vanishing viscous and Hessian-driven damping. We establish a fast sublinear convergence rate in case the objective function is convex and satisfies Polyak-\L ojasiewicz inequality. We also establish a linear convergence rate for strongly convex functions. The results can provide more insights into the convergence property of Nesterov's accelerated gradient method.     
\end{abstract}

% REQUIRED

% REQUIRED

%%%%%%%%%%%%%%%%%%%%%%%%%%%%%%%%%%%%%%%%
%%%%%%%%%%%%%%%%%%%%%%%%%%%%%%%%%%%%%%%%
\section{Introduction}
Consider the optimization problem
\begin{equation}\label{Problem: P}
\min_{x\in H} f(x),
\end{equation}
where $f:H\to\mathbb{R}$ is a twice differentiable convex function on a real Hilbert space $H$, whose gradient is $L$-Lipschitz continuous for some $L>0$. We assume that the solution set $S:=\argmin(f)$ is nonempty, and write $f^*=\min(f)$. 

Nesterov's accelerated gradient algorithm \cite{Nesterov_1983}, which may take the form:
\begin{equation}\label{E: NAG}\tag{NAG}
\left\{
\begin{array}{rcl}
y_{k+1} &=& x_k - s \nabla f(x_k),\\
x_{k+1} &=& y_{k+1} + \frac{k}{k+\alpha}(y_{k+1}-y_k),
\end{array}
\right.
\end{equation}
where $0<s\le \frac{1}{L}$ and $\alpha\ge 3$, is often used to solve problem \eqref{Problem: P}, ensuring a convergence rate of $f(y_{k+1}) - f^* \le \mathcal{O}\left( \frac{1}{k^2} \right)$ \cite{Nesterov_1983}. For $\alpha > 3$, one obtains a faster convergence rate $f(y_{k+1}) - f^* \le o\left( \frac{1}{k^2} \right)$ \cite{Attouch_2016}, and the sequence of iterates is weakly convergent \cite{Chambolle_2015}. If $f$ is convex, and satisfies a local error bound condition, an improved sublinear convergence rate $o\left( \frac{1}{k^{2\alpha}} \right)$ was derived in \cite{Liu_2024} for $\alpha>1$, $0<s<\frac{1}{L}$ and $k$ sufficiently large. Although it is not mentioned in \cite{Liu_2024}, a {\it big $\mathcal{O}$} result can be derived from their analysis. If $f$ is $\mu$-strongly convex, a linear convergence rate for the function values was shown possible in \cite{Shi_2024}, and the best rate is approximately $\mathcal{O}\left( \frac{1}{k^2}\left( 1 - \frac{\mu}{4L} \right)^k \right)$, with $s=\frac{1}{2L}$ \cite{Bao_2023}. In the subcritical case $0<\alpha<3$, \eqref{E: NAG} ensures a convergence rate of $\mathcal{O}\left( k^{-\frac{2}{3}\alpha} \right)$ if $f$ is convex \cite{Attouch_2019,Aujol_2020}. The proximal version of \eqref{E: NAG}, known as FISTA \cite{Beck_2009}, is able to solve convex composite optimization problems. In the convex case, the convergence rate is $\mathcal{O}\left( \frac{1}{k^2} \right)$ if $\alpha\ge 3$ \cite{Beck_2009}, and $o\left( \frac{1}{k^2} \right)$ if $\alpha>3$ \cite{Attouch_2016}, and otherwise $\mathcal{O}\left( k^{-\frac{2}{3}\alpha} \right)$ if $0<\alpha<3$ \cite{Attouch_2019,Aujol_2020}. When the smooth part is $\mu$-strongly convex, a linear convergence rate of up to $\mathcal{O}\left( \frac{1}{k^2}\left( 1 - \frac{\mu}{4L} \right)^k \right)$ is guaranteed for $\alpha\ge 3$ \cite{Shi_2024,Bao_2023}. To the best of our knowledge, no linear convergence rate has been found in the subcritical case. A variant of \eqref{E: NAG}, usually referred to as {\it Optimized Gradient Method} \cite{Drori_2014,Kim_2016,Kim_2017,Park_2023}, provides an improved convergence rate for convex functions. Local geometric properties of the convex objective function (quadratic growth or \L ojasiewicz inequality) also have an effect on the convergence rate of \eqref{E: NAG} \cite{Aujol_2021}.        

Nesterov's accelerated gradient algorithm \eqref{E: NAG} can be seen as a discrete analogue of an inertial system with asymptotic vanishing damping \cite{Su_2016}:
\begin{equation}\label{E: AVD}
\ddot{x}(t) + \frac{\alpha}{t}\dot{x}(t) + \nabla f\big(x(t)\big) = 0,
\end{equation}
with $\alpha\ge3$, with convergence rate $f\big(x(t)\big) - f^* \le \mathcal{O}\left( \frac{1}{t^2} \right)$ if $f$ is convex, and $f\big(x(t)\big) - f^* \le \mathcal{O}\left( t^{-\frac{2}{3}\alpha} \right)$ (with no hope for a linear rate) if $f$ is strongly convex \cite{Su_2016}. For $\alpha>3$, the trajectory convergence is guaranteed \cite{Attouch_2018}, and $f\big(x(t)\big)-f^*\le o\left(\frac{1}{t^2}\right)$ \cite{May_2017}. For $0<\alpha<3$, one obtains $f\big(x(t)\big) - f^* \le \mathcal{O}\left( t^{-\frac{2}{3}\alpha} \right)$ if $f$ is convex \cite{Attouch_2019}. Convergence rate analysis of \eqref{E: AVD} under local geometry of the convex function $f$ can be found in \cite{Aujol_2019}. However, the system \eqref{E: AVD} exhibits wild trajectory oscillations in practice, which motivates inclusion of a Hessian-driven damping term in \eqref{E: AVD} so as to attenuate the oscillations \cite{Attouch_2016_Hessian}:
\begin{equation}\label{E: AVD-H-basic}
\ddot{x}(t) + \frac{\alpha}{t}\dot{x}(t) + \beta\nabla^2 f\big(x(t)\big) \dot{x}(t) + \nabla f\big(x(t)\big) = 0,
\end{equation}
where $\alpha\ge 3$ and $\beta>0$. This system preserves most properties of \eqref{E: AVD}, including $f\big(x(t)\big) - f^* \le o\left( \frac{1}{t^2} \right)$ if $\alpha>3$, and also a faster convergence of the gradient to zero \cite{Attouch_2016_Hessian}. Under a high-resolution perspective, the following system \cite{Shi_2022}
\begin{equation}\label{E: AVD-H-Nesterov}
\ddot{x}(t) + \frac{3}{t}\dot{x}(t) + \sqrt{s}\nabla^2 f\big(x(t)\big) \dot{x}(t) + \left(1 + \frac{3\sqrt{s}}{2t} \right) \nabla f\big(x(t)\big) = 0,
\end{equation}
turns out to be the limiting ordinary-differential-equation (ODE) for \eqref{E: NAG}. If $f$ is $\mu$-strongly convex, a linear convergence rate of $f\big(x(t)\big) - f^* \le \mathcal{O}\left( \frac{1}{t^2}e^{-\frac{\mu\sqrt{s}}{4}t} \right)$, for $t\ge \frac{4}{\mu\sqrt{s}}$ and $0<s<\frac{1}{L}$, holds for the solutions of \eqref{E: AVD-H-Nesterov} \cite{Shi_2024}. A generalization of the systems \eqref{E: AVD-H-basic} and \eqref{E: AVD-H-Nesterov}, with more general coefficients, was studied in \cite{Attouch_2022}. 

In this paper, and in line with \eqref{E: AVD-H-Nesterov}, we consider the inertial system
\begin{equation}\label{System: AVD-H}\tag{AVD-H}
\ddot{x}(t) + \frac{\alpha}{t}\dot{x}(t) + \beta\nabla^2 f\big(x(t)\big)\dot{x}(t) + \left( \gamma + \frac{r}{t} \right) \nabla f\big(x(t)\big) = 0,
\end{equation}
with $\alpha,\beta,\gamma,r > 0$, and initial conditions given by $x(0) = x_0$ and $\dot{x}(0) = v_0$ (for the issue of existence and uniqueness of solution, the reader is referred to the discussion in \cite{Juan_2023}). By appropriately tuning the coefficients, we are able to establish the fastest convergence rates up to date, which are valid for all $\alpha>0$. This knowledge allows us to transfer some of these results to accelerated gradient algorithms, for which we obtain new convergence rates. 

Our results are organized in three major sections:
Section \ref{Sec: dynamics} contains the analysis of the solutions to the differential equation \eqref{System: AVD-H}. When the (not necessarily convex) objective function $f$ satisfies a Polyak-\L ojasiewicz (P\L) inequality with constant $\mu>0$, the function value gap $f\big(x(t)\big)-f^*$ behaves linearly for $t\le\frac{\alpha}{\mu\beta}$, then enters a $\mathcal O(\frac{1}{t^{2\alpha}})$ regime, and, if $f$ is convex, asymptotically behaves as $o(\frac{1}{t^{2\alpha}})$. Surprisingly, the convergence rates do not depend on $\gamma$ if one makes the natural choice $v_0=-\beta\nabla f(x_0)$, but do require $r$ to be properly tuned. If $f$ is $\mu$-strongly convex, then $f\big(x(t)\big)-f^*\le \mathcal O\big(\frac{1}{t^{2\alpha}}e^{-\frac{1}{2}\mu\beta t}\big)$. All the results in this section are valid for $\alpha>0$. 
In Section \ref{Sec: gradient_algorithms}, we analyze the convergence rate of a family of inertial gradient algorithms from a similar perspective, when $f$ is $L$-smooth. As in the continuous case, the function value gap goes through a linear regime, a $\mathcal O(\frac{1}{k^{2\alpha}})$ regime, and asymptotically behaves as $o(\frac{1}{k^{2\alpha}})$, provided $f$ is convex and satisfies a $\mu$-P\L\ inequality. In the $\mu$-strongly convex case, the convergence rate is $\mathcal O\big((1-\rho)^k\big)$, with $\rho=\frac{\mu}{4(L+\mu)}$. Finally, in Section \ref{Sec: proximal_algorithms}, we study a family of accelerated proximal-gradient algorithms, and establish a linear convergence rate under strong convexity.

%%%%%%%%%%%%%%%%%%%%%%%%%%%%%%%%%%%%%%%%
%%%%%%%%%%%%%%%%%%%%%%%%%%%%%%%%%%%%%%%%
\section{Convergence rates I: inertial dynamics}\label{Sec: dynamics}
In this section, we prove convergence rates of the function values along the trajectories defined by the inertial system \eqref{System: AVD-H}. Define the energy function $\varepsilon:(0,\infty)\to\mathbb{R}$ by
\begin{equation}\label{E: epsilon}
\varepsilon(t) = \frac{1}{2}\left\| \dot{x} + \beta\nabla f(x) \right\|^2 + \frac{\omega}{2}\| \dot{x} \|^2 + \theta ( f(x) - f^* ),
\end{equation}
with $\theta,\omega\in\mathbb{R}$. The values of $\omega$ and $\theta$ will be fixed later, but note that, if $\theta>0$ and $\omega\ge 0$, then $f(x)-f^*\le \frac{\varepsilon}{\theta}$. We have the following:

\begin{proposition} \label{P:dot_e}
Consider the system \eqref{System: AVD-H} with $\alpha,\beta,\gamma>0$ and $r = \frac{\alpha\beta}{1+\omega}$. Take any $\eta>0$, set $\theta = (1+\omega) \gamma - \frac{\omega\eta\beta^2}{1+\omega}$, and let $\varepsilon$ be defined by \eqref{E: epsilon}. For every $t>0$, we have
\begin{align*} 
\dot{\varepsilon}
&= -\frac{2\alpha}{t} \left( \frac{1}{2}\left\| \dot{x} + \beta\nabla f(x) \right\|^2 + \frac{\omega}{2} \| \dot{x} \|^2 \right)
   -\frac{\omega\eta\beta}{1+\omega}\left( \frac{1}{2}\| \dot{x} + \beta\nabla f(x) \|^2
   + \frac{\omega}{2}\| \dot{x} \|^2 \right) \\
&\quad - \frac{\beta}{1+\omega}\left( \theta + \frac{\omega(1-\omega)}{2(1+\omega)}\eta\beta^2 - \frac{\omega\alpha\beta}{t} \right) \left\| \nabla f(x) \right\|^2
   - \omega\beta\left\langle \nabla^2 f(x)\dot{x}, \dot{x} \right\rangle
   + \frac{\omega\eta\beta}{2}\| \dot{x} \|^2.
\end{align*}   
\end{proposition}

\begin{proof}
We begin by using \eqref{System: AVD-H}, to compute the time derivative of $\varepsilon$:
\begin{align*}
\dot{\varepsilon}
&= \left\langle \dot{x} + \beta\nabla f(x), \ddot{x} + \beta\nabla^2 f(x)\dot{x} \right\rangle + \omega\langle \dot{x}, \ddot{x} \rangle + \theta\langle \nabla f(x), \dot{x} \rangle \\
&= \left\langle \dot{x} + \beta\nabla f(x), -\frac{\alpha}{t}\dot{x} - \left( \gamma + \frac{r}{t} \right)\nabla f(x) \right\rangle
  + \theta\langle \nabla f(x), \dot{x} \rangle \\
&\quad + \omega\left\langle \dot{x}, -\frac{\alpha}{t}\dot{x} - \beta\nabla^2 f(x)\dot{x} - \left( \gamma + \frac{r}{t} \right)\nabla f(x)  \right\rangle \\
&= - \frac{\alpha}{t}\left\| \dot{x} + \beta\nabla f(x) \right\|^2
   - \left[ (1+\omega) \gamma - \theta + \frac{(1+\omega)r-\alpha\beta}{t} \right] \langle \nabla f(x), \dot{x} \rangle \\
&\quad - \beta\left( \gamma + \frac{r-\alpha\beta}{t} \right) \left\| \nabla f(x) \right\|^2 
    - \frac{\omega\alpha}{t}\| \dot{x} \|^2 
    - \omega\beta\left\langle \nabla^2 f(x)\dot{x}, \dot{x} \right\rangle.
\end{align*}
Substituting $r = \frac{\alpha\beta}{1+\omega}$, and then $\theta = (1+\omega) \gamma - \frac{\omega\eta\beta^2}{1+\omega}$, we obtain
\begin{align*} 
\dot{\varepsilon}
&= -\frac{2\alpha}{t} \left( \frac{1}{2}\left\| \dot{x} + \beta\nabla f(x) \right\|^2 + \frac{\omega}{2} \| \dot{x} \|^2 \right)
   -[(1+\omega) \gamma - \theta] \langle \nabla f(x), \dot{x} \rangle \\
&\quad - \beta\left( \gamma - \frac{\omega\alpha\beta}{(1+\omega)t} \right) \left\| \nabla f(x) \right\|^2
   - \omega\beta\left\langle \nabla^2 f(x)\dot{x}, \dot{x} \right\rangle \\
& = -\frac{2\alpha}{t} \left( \frac{1}{2}\left\| \dot{x} + \beta\nabla f(x) \right\|^2 + \frac{\omega}{2} \| \dot{x} \|^2 \right)
   -\frac{\omega\eta\beta}{1+\omega}\Big[\beta \langle \nabla f(x), \dot{x} \rangle\Big] \\
&\quad - \frac{\beta}{1+\omega}\left( \theta + \frac{\omega\eta\beta^2}{1+\omega} - \frac{\omega\alpha\beta}{t} \right) \left\| \nabla f(x) \right\|^2
   - \omega\beta\left\langle \nabla^2 f(x)\dot{x}, \dot{x} \right\rangle.
\end{align*}
We conclude by observing that
\begin{align*}
\beta \langle \nabla f(x), \dot{x} \rangle 
&= \frac{1}{2}\| \dot{x} + \beta\nabla f(x) \|^2
   - \frac{1}{2}\| \dot{x} \|^2
   - \frac{1}{2}\beta^2 \| \nabla f(x) \|^2  \\
&= \frac{1}{2}\| \dot{x} + \beta\nabla f(x) \|^2
   + \frac{\omega}{2}\| \dot{x} \|^2
   - \frac{1+\omega}{2}\| \dot{x} \|^2
   - \frac{\beta^2 }{2}\| \nabla f(x) \|^2,
\end{align*}
and then rearranging the terms.
\end{proof}

%%%%%%%%%%%%%%%%%%%%%%%%%%%%%%%%%%%%%%%%
%%%%%%%%%%%%%%%%%%%%%%%%%%%%%%%%%%%%%%%%
\subsection{Polyak-\L ojasiewicz inequality}
We now analyze the convergence rate of the function values when the objective function satisfies Polyak-\L ojasiewicz inequality with constant $\mu >0$, namely:
\begin{equation} \label{E: PL}
\| \nabla f(x) \|^2 \ge 2\mu( f(x) - f^* ).    
\end{equation}

We have the following:

\begin{theorem}\label{Thm: dynamics_PL}
Let $f$ be twice differentiable and satisfy Polyak-\L ojasiewicz inequality with constant $\mu > 0$. Consider the system \eqref{System: AVD-H}, where $\alpha,\beta,\gamma>0$ and $r= \alpha\beta$. Let $v_0 = -\beta\nabla f(x_0)$. For every $t>0$, we have  
\begin{equation*}
f(x) - f^* \le\left\{
\begin{array}{ll}
\left( f(x_0) - f^* \right) e^{-2\mu\beta t},&\quad t\le \frac{\alpha}{\mu\beta}, \\[5pt]
\left( f(x_0) - f^* \right) \left( \frac{\alpha}{\mu\beta e} \right)^{2\alpha}\frac{1}{t^{2\alpha}},&\quad t > \frac{\alpha}{\mu\beta}.
\end{array}
\right.
\end{equation*}
Moreover, if $f$ is convex, then $f\big( x(t) \big) - f^* \le o\left( \frac{1}{t^{2\alpha}} \right)$ as $t\to\infty$.
\end{theorem}

\begin{proof}
Set $\omega=0$ and $\eta=\mu$ in Proposition \ref{P:dot_e}, and use the Polyak-\L ojasiewicz inequality \eqref{E: PL} to obtain
$$\dot{\varepsilon} 
= - \frac{\alpha}{t}\left\| \dot{x} + \beta\nabla f(x) \right\|^2 - \beta\theta \left\| \nabla f(x) \right\|^2
\le - \frac{\alpha}{t}\left\| \dot{x} + \beta\nabla f(x) \right\|^2
- 2\mu\beta\gamma (f(x) - f^*),$$
($\theta=\gamma>0$). It follows that $ \dot{\varepsilon} \le - 2\min\left\{ \frac{\alpha}{t}, \mu\beta \right\} \varepsilon$, which results in
\begin{equation*}
\varepsilon(t) \le \left\{
\begin{array}{rcl}
\varepsilon(0) e^{-2\mu\beta t}, & & t \le T:=\frac{\alpha}{\mu\beta},\\
\frac{\varepsilon(T)}{ \left( t/T \right)^{2\alpha} }, & & t > T. 
\end{array}
\right.
\end{equation*}
If $v_0 = -\beta\nabla f(x_0)$, we have $\varepsilon(0) = \gamma\left( f(x_0) - f^* \right)$ and $$\varepsilon(T) \le \varepsilon(0) e^{-2\mu\beta T} = \gamma\left( f(x_0) - f^* \right) e^{-2\alpha}.$$
The bound is obtained by noting that $f(x) - f^* \le \frac{\varepsilon(t)}{\gamma}$. 

To prove that $f\big(x(t)\big) - f^* \le o\left( \frac{1}{t^{2\alpha}} \right)$ as $t\to\infty$, we can assume $t > 2 T$, so that
$$\dot{\varepsilon} 
\le - \frac{\alpha}{t}\left\| \dot{x} + \beta\nabla f(x) \right\|^2
- 2\mu\beta\gamma (f(x) - f^*)
\le - \frac{2\alpha}{t}\varepsilon - \mu\beta\gamma( f(x) - f^* ). $$
Multiplying by $t^{2\alpha}$ gives
$$ \mu\beta\gamma t^{2\alpha}( f(x) - f^* ) \le -\frac{d}{dt}\left( t^{2\alpha}\varepsilon \right). $$
Integrating from $2T$ to $t$ results in
$$ \mu\beta\gamma \int_{2T}^{t} \tau^{2\alpha} ( f(x(\tau)) - f^* )\,d\tau \le -  
\left. \tau^{2\alpha}\varepsilon(\tau)\right|_{\tau=2T}^t
\le (2T)^{2\alpha}\varepsilon(2T) < \infty. $$
Write $g(t) = \beta t^{2\alpha}( f(x) - f^* )$, so that $g(t)\ge 0$ and $\int_{0}^{\infty} g(\tau)\,d\tau < \infty$. From
$$ \dot{g}(t) = 2\alpha\beta t^{2\alpha-1}( f(x) - f^* ) + t^{2\alpha}\langle \beta \nabla f(x), \dot{x} \rangle,$$
it follows that
$$| \dot{g}(t) | 
\le  2\alpha\beta t^{2\alpha-1}( f(x) - f^* ) + t^{2\alpha} \left( \frac{ \| \dot{x} + \beta\nabla f(x) \|^2 + 3\| \beta\nabla f(x) \|^2  }{2} \right).$$
But, $\varepsilon\le\mathcal O(\frac{1}{t^{2\alpha}})$, $\| \dot{x} + \beta\nabla f(x) \|^2\le\mathcal O(\frac{1}{t^{2\alpha}})$ and $\frac{1}{2L}\| \nabla f(x) \|^2\le f(x) - f^*\le\mathcal O(\frac{1}{t^{2\alpha}})$. This shows that $|\dot g|$ is bounded on $(0,\infty)$. Since $g$ is integrable, this means that $\lim_{t\to\infty} g(t) = 0 $, and thus $f(x) - f^* \le o\left( \frac{1}{t^{2\alpha}}  \right)$, as desired.
\end{proof}

\begin{remark}
For these kinds of systems, the best results so far were $\mathcal{O}\left( t^{-\frac{2}{3}\alpha} \right)$, which were obtained for 
$$ \ddot{x} + \frac{\alpha}{t}\dot{x} + \beta\nabla^2 f(x)\dot{x} + \nabla f(x) = 0,\quad\alpha\ge 3,\ \beta\ge  0,$$
assuming strong convexity \cite{Su_2016,Attouch_2016_Hessian,Attouch_2018}. We have established a faster convergence rate, namely $\mathcal{O}\left( t^{-2\alpha} \right)$, which is valid for all $\alpha>0$ and only assuming the P\L\ inequality, even without convexity.
\end{remark}

\begin{remark}
The turning point from the linear regime to $\mathcal{O}\left( t^{-2\alpha} \right)$ occurs after the function value gap has been reduced by a factor of $e^{2\alpha}$. This takes place earlier if $\mu$ is larger.
\end{remark}

%%%%%%%%%%%%%%%%%%%%%%%%%%%%%%%%%%%%%%%%
%%%%%%%%%%%%%%%%%%%%%%%%%%%%%%%%%%%%%%%%
\subsection{Strong convexity}
In this subsection, we analyze the convergence rate of the function values when the objective function is strongly convex. Although {\it it is not possible} \cite{Su_2016} to obtain a linear rate if $\beta=0$, we do find one when $\beta>0$. 

\begin{theorem}
Let $f$ be twice differentiable and $\mu$-strongly convex. Consider the system \eqref{System: AVD-H}, where, 
$$ \alpha > 0,\quad \beta > 0,\quad \gamma>\frac{1}{4}\mu\beta^2,\quad r= \frac{1}{2}\alpha\beta. $$
For every $t\ge T:=\frac{4\alpha}{\mu\beta} + \frac{4\alpha\beta}{4\gamma-\mu\beta^2}$, we have, 
$$ f(x) - f^* \le \frac{e^{-\frac{1}{2}\mu\beta (t-T)}}{\left( t/T \right)^{2\alpha}} C_T, $$
with $C_T = \frac{1}{2}\| v(T) + \beta\nabla f(x(T)) \|^2 + \frac{1}{2}\| v(T) \|^2 + \left(2\gamma-\frac{1}{2}\mu\beta^2\right)\big( f(x(T)) - f^* \big)$.
\end{theorem}

\begin{proof}
Setting $\omega=1$ and $\eta=\mu$, we have $\varepsilon = \frac{1}{2}\left\| \dot{x} + \beta\nabla f(x) \right\|^2 + \frac{1}{2}\| \dot{x} \|^2 + \theta (f(x) - f^*)$, where $\theta = 2\gamma - \frac{1}{2}\mu\beta^2>0$. Since $f$ is $\mu$-strongly convex, we have $\left\langle \nabla^2 f(x)\dot{x}, \dot{x} \right\rangle \ge \mu\| \dot{x} \|^2$. Using Proposition \ref{P:dot_e}, we obtain
\begin{align*}
\dot{\varepsilon}
&\le -\frac{2\alpha}{t} \left( \frac{1}{2}\left\| \dot{x} + \beta\nabla f(x) \right\|^2 + \frac{1}{2} \| \dot{x} \|^2 \right) 
   - \frac{\mu\beta}{2} \left( \frac{1}{2}\left\| \dot{x} + \beta\nabla f(x) \right\|^2 + \frac{1}{2} \| \dot{x} \|^2 \right) \\
&\quad - \mu\beta\left( \theta - \frac{\alpha\beta}{t} \right) \left( f(x) - f^* \right) \\
& = -\left(\frac{2\alpha}{t}+\frac{\mu\beta}{2}\right) \varepsilon  - \mu\beta\left( \frac{\theta}{2} - \frac{\alpha\beta + \frac{2\alpha\theta}{\mu\beta}}{t} \right) ( f(x) - f^* )\\
&\le -\left(\frac{2\alpha}{t}+\frac{\mu\beta}{2}\right) \varepsilon, 
\end{align*}
where the last inequality is due to $t\ge \frac{4\alpha}{\mu\beta} + \frac{4\alpha\beta}{4\gamma-\mu\beta^2}=T$. This immediately gives the result.
\end{proof}

\begin{remark}
A convergence rate of 
$f(x) - f^* \le \mathcal{O}\left( \frac{1}{t^{2\alpha}}e^{-\frac{1}{2}\mu\beta t} \right)$ holds for the system
$$ \ddot{x} + \frac{\alpha}{t}\dot{x} + \beta\nabla^2 f(x)\dot{x} + \left( 1 + \frac{\alpha\beta}{2t} \right) \nabla f(x) = 0 ,$$
when $\alpha>0$, $0<\beta<\frac{2}{\sqrt{\mu}}$ and $t$ is large enough. The case $\alpha=3$ was studied in \cite{Shi_2024}. Assuming $0<\beta<\frac{1}{\sqrt{L}}$, they obtain a convergence rate of $\mathcal{O}\left( \frac{1}{t^2}e^{-\frac{1}{4}\mu\beta t} \right)$. The exponent in the linear part is half as large as ours, and that of the sublinear part is just one third of ours, all under a stronger constraint for $\beta$.
\end{remark}

%%%%%%%%%%%%%%%%%%%%%%%%%%%%%%%%%%%%%%%%
%%%%%%%%%%%%%%%%%%%%%%%%%%%%%%%%%%%%%%%%
\section{Convergence rates II: gradient algorithms}\label{Sec: gradient_algorithms}
In this section, we develop accelerated gradient algorithms and prove their convergence rates for the function values. The algorithms are obtained by discretizing the dynamics \eqref{System: AVD-H} and attain different rates depending on the geometric characteristics of the objective function.

Let $f$ be $L$-smooth, define the step size $h\in\left(0,\frac{1}{\sqrt{L}}\right]$ and introduce the velocity
$$ v_k := \frac{x_{k+1}-x_k}{h}. $$
The term involving Hessian-driven damping is approximated by
$$\nabla^2 f(x)\dot{x} \sim \frac{\nabla f(x_k) - \nabla f(x_{k-1})}{h}.$$
We combine this with
\begin{equation*}
\left\{
\begin{array}{ccl}
t &\sim& kh,\\
\ddot{x} &\sim& \frac{v_k - v_{k-1}}{h},\\
\dot{x} &\sim& v_k,\\
\nabla f(x) &\sim& \nabla f(x_k), 
\end{array}
\right.
\end{equation*}
and set $\beta=h$ and $r=\alpha\beta=\alpha h$, to obtain
\begin{equation}\label{inertial_iterates}
( v_k - v_{k-1} ) + \frac{\alpha}{k}v_k + h\left( \nabla f(x_k) - \nabla f(x_{k-1}) \right) + \left( \gamma + \frac{\alpha}{k} \right) h \nabla f(x_k) = 0.
\end{equation}
Setting $y_{k+1} = x_k - h^2\nabla f(x_k)$, we can rewrite \eqref{inertial_iterates} as
\begin{equation}\label{Algo: AGM}\tag{AGM}
\left\{\begin{array}{ccl}
y_{k+1} &=& x_k - h^2\nabla f(x_k) \\
x_{k+1} &=& y_{k+1} + \frac{k}{k+\alpha}\left( y_{k+1} - y_k \right) + \frac{k}{k+\alpha}(\gamma-1) ( y_{k+1} - x_k ).
\end{array}
\right.
\end{equation}
The convergence rates for this algorithm will be investigated in what follows.

%%%%%%%%%%%%%%%%%%%%%%%%%%%%%%%%%%%%%%%%
%%%%%%%%%%%%%%%%%%%%%%%%%%%%%%%%%%%%%%%%
\subsection{General estimations}
In this subsection, we derive some useful estimations for proving the convergence rate of the algorithm \eqref{Algo: AGM}. Our convergence proofs are based on the energy-like sequence $(E_k)_{k\ge 1}$, defined by
\begin{equation}\label{E: E_k}
E_k = \frac{1}{2}\| \phi_k \|^2 + \gamma\left[ (1-\omega)\psi_k + \omega\left( f(y_{k+1})-f^* \right) \right],
\end{equation}
where $0\le\omega\le 1$, and
\begin{align*}
\phi_k &:= v_k + h\nabla f(x_k),\\
\psi_k &:= f(x_k) - f^* - \frac{h^2}{2}\| \nabla f(x_k) \|^2.
\end{align*}

\begin{remark}
Notice that $y_{k+1} = x_k - h^2\nabla f(x_k)$ in \eqref{Algo: AGM}. Since $f$ is $L$-smooth and $h^2\le \frac{1}{L}$, we obtain
$$ f(y_{k+1}) \le f(x_k) - \frac{h^2}{2}\| \nabla f(x_k) \|^2.$$
This gives
$$ 0 \le f(y_{k+1}) - f^*  \le \psi_k,\qqbox{and} 0 \le f(y_{k+1}) - f^*  \le \frac{1}{\gamma}E_k.$$
\end{remark}

We have the following:

\begin{proposition} \label{P:Delta_Ek}
Let $f$ be convex and $L$-smooth, and set $0<h\le \frac{1}{\sqrt{L}}$. Consider the sequence $(E_k)_{k\ge 1}$ defined by \eqref{E: E_k}, where $0\le\omega\le 1$. We have  
\begin{align*}
E_{k+1} - E_k 
=&    -\tfrac{\alpha}{k+1}\left( 1 + \tfrac{\alpha/2}{k+1} \right)\| v_{k+1} + h\nabla f(x_{k+1}) \|^2
   -\gamma h^2 \left(1-\tfrac{\gamma}{2}\right)\| \nabla f(x_{k+1}) \|^2 \\
   & +\omega\gamma \left(f(x_{k+1}) - f(y_{k+1})  \right) 
   -\omega\gamma h\left(\tfrac{k+\alpha+1}{k+1}\right) \langle \nabla f(x_{k+1}), v_{k+1} \rangle \\
&  -\omega\gamma h^2 \left(1+\gamma+\tfrac{\alpha}{k+1}-\tfrac{L h^2}{2}\right)\| \nabla f(x_{k+1}) \|^2  .
\end{align*}
\end{proposition}

\begin{proof}
By definition of $(E_k)_{k\ge 1}$, we have
\begin{align} \label{E:Energy_diff}
E_{k+1} - E_k = & \left( \frac{1}{2}\| \phi_{k+1} \|^2 - \frac{1}{2}\| \phi_k \|^2 \right) \nonumber \\
& + (1-\omega)\gamma (\psi_{k+1} - \psi_k) + \omega\gamma ( f(y_{k+2}) - f(y_{k+1}) ).    
\end{align}
From \eqref{inertial_iterates}, we obtain
\begin{equation}\label{E: iterates}
( v_{k+1} - v_k ) + h\left( \nabla f(x_{k+1}) - \nabla f(x_k) \right)
= -\frac{\alpha}{k+1} v_{k+1} - \left( \gamma + \frac{\alpha}{k+1} \right) h\nabla f(x_{k+1}).
\end{equation}
As a result,
$$ \phi_{k+1} - \phi_k = -\frac{\alpha}{k+1} \left( v_{k+1} + h\nabla f(x_{k+1}) \right) - \gamma h\nabla f(x_{k+1}), $$
so that
\begin{align*}
\| \phi_{k+1} - \phi_k \|^2 
&= \left( \frac{\alpha}{k+1} \right)^2 \| v_{k+1} + h\nabla f(x_{k+1}) \|^2
  + \left( \gamma + \frac{2\alpha}{k+1} \right)\gamma h^2 \| \nabla f(x_{k+1}) \|^2 \\
&\quad + \frac{2\alpha}{k+1}\gamma h \langle \nabla f(x_{k+1}), v_{k+1} \rangle,
\end{align*}
and
\begin{align*}
&\quad \langle \phi_{k+1}, \phi_{k+1} - \phi_k \rangle \\
&= \left\langle v_{k+1} + h\nabla f(x_{k+1}), -\frac{\alpha}{k+1} \left( v_{k+1} + h\nabla f(x_{k+1}) \right) - \gamma h\nabla f(x_{k+1}) \right\rangle \\
&= -\frac{\alpha}{k+1}\| v_{k+1} + h\nabla f(x_{k+1}) \|^2
   -\gamma h \langle \nabla f(x_{k+1}), v_{k+1} \rangle
   -\gamma h^2 \| \nabla f(x_{k+1}) \|^2.
\end{align*}
Since $\frac{1}{2}\| \phi_{k+1} \|^2 - \frac{1}{2}\| \phi_k \|^2 = \langle \phi_{k+1}, \phi_{k+1} - \phi_k \rangle - \frac{1}{2}\| \phi_{k+1} - \phi_k \|^2$, it follows that
\begin{align}\label{E: phi_diff}
\frac{1}{2}\| \phi_{k+1} \|^2 - \frac{1}{2}\| \phi_k \|^2 
&= -\frac{\alpha}{k+1}\left( 1 + \frac{\alpha/2}{k+1} \right)\| v_{k+1} + h\nabla f(x_{k+1}) \|^2 \nonumber \\
&\quad    -\gamma h\left(\frac{k+\alpha+1}{k+1}\right) \langle \nabla f(x_{k+1}), v_{k+1} \rangle  \nonumber \\
&\quad  -\gamma h^2\left( 1 + \frac{1}{2}\gamma + \frac{\alpha}{k+1} \right) \| \nabla f(x_{k+1}) \|^2.
\end{align}
With $h\le\frac{1}{\sqrt{L}}$ in mind, we use the convexity and $L$-smoothness of $f$ to get
$$ f(x_{k+1}) - f(x_k) \le h \langle \nabla f(x_{k+1}), v_k \rangle - \frac{h^2}{2}\| \nabla f(x_{k+1}) - \nabla f(x_k) \|^2. $$
Using \eqref{E: iterates}, we obtain
$$ v_k = \frac{k+\alpha+1}{k+1} v_{k+1} + h\left( \nabla f(x_{k+1}) - \nabla f(x_k) \right) + \left( \gamma + \frac{\alpha}{k+1} \right) h \nabla f(x_{k+1}), $$
and then
\begin{align*}
f(x_{k+1}) - f(x_k)
&\le \frac{k+\alpha+1}{k+1} h \langle \nabla f(x_{k+1}), v_{k+1} \rangle
     + \frac{h^2}{2}\| \nabla f(x_{k+1}) \|^2 \\
&\quad  - \frac{h^2}{2}\| \nabla f(x_k) \|^2 
     + \left( \gamma + \frac{\alpha}{k+1} \right) h^2 \| \nabla f(x_{k+1}) \|^2.
\end{align*}
This gives
\begin{equation}\label{E: psi_diff}
\psi_{k+1} - \psi_k 
\le \frac{k+\alpha+1}{k+1} h \langle \nabla f(x_{k+1}), v_{k+1} \rangle
   + \left( \gamma + \frac{\alpha}{k+1} \right) h^2 \| \nabla f(x_{k+1}) \|^2.
\end{equation}
In view of $y_{k+2} = x_{k+1} - h^2\nabla f(x_{k+1})$ and $L$-smoothness of $f$, we have
$$ f(y_{k+2}) \le f(x_{k+1}) - h^2\left( 1 - \frac{1}{2}Lh^2 \right)\| \nabla f(x_{k+1}) \|^2, $$
from which we obtain
\begin{equation} \label{E:f_y_k+1}
f(y_{k+2}) - f(y_{k+1}) \le f(x_{k+1}) - f(y_{k+1}) - h^2\left( 1 - \frac{1}{2}Lh^2 \right)\| \nabla f(x_{k+1}) \|^2.  
\end{equation}
To conclude, we substitute \eqref{E: phi_diff}, \eqref{E: psi_diff} and \eqref{E:f_y_k+1} into \eqref{E:Energy_diff}, and rearrange the terms.
\end{proof}

%%%%%%%%%%%%%%%%%%%%%%%%%%%%%%%%%%%%%%%%
%%%%%%%%%%%%%%%%%%%%%%%%%%%%%%%%%%%%%%%%
\subsection{Convexity and Polyak-\L ojasiewicz inequality}
In this subsection, we prove the convergence rate of the function values for the algorithm \eqref{Algo: AGM} in case the objective function is convex and satisfies Polyak-\L ojasiewicz inequality.

\begin{theorem}\label{Thm: algo_PL}
Let $f$ be convex, $L$-smooth and satisfy Polyak-\L ojasiewicz inequality with constant $\mu$, where $L>\mu>0$. Generate the sequences $(x_k)$ and $(y_k)$ according to \eqref{Algo: AGM} with $\alpha>0$, $\gamma\in(0,2)$ and $0< h\le\frac{1}{\sqrt{L}}$.
%$$ \alpha>0,\quad \gamma\in(0,2),\quad 0< h\le\frac{1}{\sqrt{L}},\quad r=\alpha h. $$
Given $x_1$, set $v_1 = -h\nabla f(x_1)$. We have
\begin{equation*}
f(y_{k+1}) - f^*\le\left\{
\begin{array}{ll}
\frac{ f(x_1) - f^* }{ \left[ 1 + \mu h^2 (2-\gamma) \right]^{k-1} },&\quad 1\le k \le K,\\ [5pt]
\frac{f(x_1) - f^*}{\left[ 1 + \mu h^2 (2-\gamma) \right]^{K-1} } \left( \frac{K+1+\alpha}{k+1+\alpha} \right)^{2\alpha},&\quad k> K,
\end{array}
\right.
\end{equation*}
where $K$ is the largest integer such that $\mu h^2(2-\gamma)K^2-2\alpha K-\alpha^2\le 0$. Moreover, $f(y_{k+1}) - f^* \le o( \frac{1}{k^{2\alpha}})$ as $k\to\infty$.
\end{theorem}

\begin{proof}
Setting $\omega=0$, we have $E_{k+1} = \frac{1}{2}\| v_{k+1} + h\nabla f(x_{k+1}) \|^2 + \gamma\psi_{k+1}$. Using Proposition \ref{P:Delta_Ek} and the Polyak-\L ojasiewicz inequality, we obtain
\begin{align*}
E_{k+1} - E_k
&\le -\frac{\alpha}{k+1}\left( 1 + \tfrac{\alpha/2}{k+1} \right) \| v_{k+1} + h\nabla f(x_{k+1}) \|^2
     -\mu h^2(2-\gamma)\gamma \left( f(x_{k+1}) - f^* \right)\\
&\le -\frac{\alpha}{k+1}\left( 1 + \frac{\alpha/2}{k+1} \right) \| v_{k+1} + h\nabla f(x_{k+1}) \|^2
     -\mu h^2(2-\gamma)\gamma \psi_{k+1} \\
&\le -\min\left\{ \frac{2\alpha}{k+1} + \frac{\alpha^2}{(k+1)^2}, \mu h^2(2-\gamma) \right\} E_{k+1}.
\end{align*}
If $1\le k \le K$, we have
$$ E_k \le \frac{E_1}{ \left[ 1 + \mu h^2 (2-\gamma) \right]^{k-1} }. $$
If $k>K$, we obtain $(1+\frac{\alpha}{k})^2E_k \le E_{k-1}$, which, applied recursively, gives 
$$ E_k \le  \frac{E_K}{\prod_{m=K+1}^{k}\left( 1 + \frac{\alpha}{m} \right)^2}. $$
Notice that
\begin{align*}
\prod_{m=K+1}^{k}\left( 1 + \frac{\alpha}{m} \right)^2
&= \exp\left[ 2\sum_{m=K+1}^{k}\ln\left( 1 + \frac{\alpha}{m} \right) \right] \\
&\ge \exp\left[ 2\alpha\sum_{m=K+1}^{k}\frac{1}{m+\alpha} \right] \\
&\ge  \exp\left[ 2\alpha \ln\left(\frac{k+1+\alpha}{K+1+\alpha}\right) \right],
\end{align*}
where the first inequality is due to $\ln(1+x)\ge\frac{x}{x+1}$ for $x>0$, and the second inequality is due to $\sum_{m=K+1}^{k}\frac{1}{m+\alpha} \ge  \int_{K+1}^{k+1} \frac{1}{x+\alpha}\,d x $. Hence,
$$ E_k \le \left( \frac{K+1+\alpha}{k+1+\alpha} \right)^{2\alpha}E_K. $$
With $v_1 = -h\nabla f(x_1)$, we have
$$E_1 \le \gamma\left( f(x_1) - f^* \right),$$
and
$$E_K \le \frac{E_1}{ \left[ 1 + \mu h^2 (2-\gamma) \right]^{K-1} } \le \frac{\gamma\left( f(x_1) - f^* \right)}{ \left[ 1 + \mu h^2 (2-\gamma) \right]^{K-1} }.$$
We conclude by recalling that $f(y_{k+1}) - f^* \le \frac{E_k}{\gamma}$.

To prove that $f(y_{k+1}) - f^* \le o\left( k^{-2\alpha} \right)$ as $k\to\infty$, we use Proposition \ref{P:Delta_Ek} (recall that $\omega=0$) to show that
\begin{align*}
E_{k+1} - E_k \le &
    -\frac{2\alpha}{k+1}E_{k+1} +\frac{2\alpha\gamma}{k+1}\psi_{k+1} -\gamma h^2 \left(1-\frac{\gamma}{2}\right)\| \nabla f(x_{k+1}) \|^2.
\end{align*}
Using the Polyak-\L ojasiewicz inequality and the fact that $\psi_{k+1}\le f(x_{k+1})-f^*$, we obtain
$$\left(1+\frac{2\alpha}{k+1}\right)E_{k+1} - E_k \le
    \left(\frac{2\alpha\gamma}{k+1}-\mu\gamma h^2 (2-\gamma)\right)\big(f(x_{k+1})-f^*\big).$$
Since we are concerned with the behavior as $k\to\infty$, we can assume, without loss of generality, that $\frac{2\alpha}{k+1}\le \frac{1}{2}\mu h^2(2-\gamma)$, so that
$$\left(1+\frac{2\alpha}{k+1}\right)E_{k+1} - E_k 
+ \mu\gamma h^2\left(1-\frac{\gamma}{2}\right)\big(f(x_{k+1})-f^*\big)\le 0.$$
Multiplying by $(k+1+2\alpha)^{2\alpha}$, we get
$$(k+1+2\alpha)^{2\alpha}\left(1+\frac{2\alpha}{k+1}\right) E_{k+1} - (k+1+2\alpha)^{2\alpha}E_k +\delta_k \le 0,$$
where $\delta_k=\mu\gamma h^2\left(1-\frac{\gamma}{2}\right)(k+1+2\alpha)^{2\alpha}\big(f(x_{k+1})-f^*\big)$. If we show that
\begin{equation} \label{E:k+1+alpha}
(k+1+2\alpha)^{2\alpha}\left(1+\frac{2\alpha}{k+1}\right) \ge (k+2+2\alpha)^{2\alpha},
\end{equation}
it will follow that $\sum_{k=0}^\infty\delta_k<\infty$, and so $\lim_{k\to\infty}\delta_k=0$, which is the same as $f(x_{k+1})-f^*\le o(\frac{1}{k^{2\alpha}})$. To this end, define 
$$\xi(t)=\left(1+\frac{2\alpha}{t}\right)\left(\frac{t+2\alpha}{t+1+2\alpha}\right)^{2\alpha},$$
so that
\begin{align*}
\dot{\xi}(t) &= -\frac{2\alpha}{t^2}\left(\frac{t+2\alpha}{t+1+2\alpha}\right)^{2\alpha}
+\frac{2\alpha}{t(t+1+2\alpha)}\left(\frac{t+2\alpha}{t+1+2\alpha}\right)^{2\alpha} \\
& = \frac{2\alpha}{t}\left(\frac{t+2\alpha}{t+1+2\alpha}\right)^{2\alpha}\left(\frac{1}{t+1+2\alpha}-\frac{1}{t}\right).    
\end{align*}
Clearly, $\dot{\xi}(t)<0$ for every $t>0$. Since $\lim_{t\to\infty}\xi(t)=1$, we must have $\xi(t)\ge 1$ for every $t>0$. This implies that \eqref{E:k+1+alpha} holds, and completes the proof.
\end{proof}

\begin{remark}
Setting $\gamma=1$ in \eqref{Algo: AGM}, one obtains $f(y_{k+1}) - f^* \le o\left( k^{-2\alpha} \right)$, for the algorithm 
\begin{equation*}
\left\{\begin{array}{ccl}
y_{k+1} &=& x_k - h^2\nabla f(x_k) \\
x_{k+1} &=& y_{k+1} + \frac{k}{k+\alpha}\left( y_{k+1} - y_k \right),
\end{array}
\right.
\end{equation*}
where $\alpha>0$ and $0<h\le\frac{1}{\sqrt{L}}$. A similar convergence rate $o\left( k^{-2\alpha} \right)$ was derived in \cite{Liu_2024}, with $\alpha>1$ and $0<h<\frac{1}{\sqrt{L}}$, for convex functions satisfying a local error bound condition, which is equivalent to Polyak-\L ojasiewicz inequality, for convex functions. Our approach extends the result to all $\alpha>0$ and also includes the critical step size $h=\frac{1}{\sqrt{L}}$.     
\end{remark}

%%%%%%%%%%%%%%%%%%%%%%%%%%%%%%%%%%%%%%%%
%%%%%%%%%%%%%%%%%%%%%%%%%%%%%%%%%%%%%%%%
\subsection{Strong convexity}
We now turn our attention to the strongly convex case.

\begin{theorem} \label{T:grad_SC}
Let $f$ be $\mu$-strongly convex and $L$-smooth. Generate $(x_k)$ and $(y_k)$ according to \eqref{Algo: AGM} with $\alpha>0$, $h\in(0,\frac{1}{\sqrt{L}}]$ and $\gamma\in(0,2-Lh^2)$. Given $x_1$, set $v_1 = -h\nabla f(x_1)$. For every $k\ge 1$, we have
$$ f(y_{k+1}) - f^* \le \left( 1 - \rho \right)^{k-1} \left( f(x_1)-f^* \right), $$
where $\rho \ge \mu \min\left\{ \frac{\gamma h^2}{1 + 2L\gamma h^2},  \frac{h^2(2-\gamma - Lh^2)}{1+\frac{\mu}{L}} \right\}$.
\end{theorem}

\begin{proof}
Setting $\omega=1$ in Proposition \ref{P:Delta_Ek}, we get
\begin{align} \label{E:G_SC_1}
E_{k+1} - E_k 
\le & \   \gamma \left(f(x_{k+1}) - f(y_{k+1})  \right)   -\gamma h\left(\frac{k+\alpha+1}{k+1}\right) \langle \nabla f(x_{k+1}), v_{k+1} \rangle  \nonumber \\
&  -\gamma h^2 \left(2+\frac{\gamma}{2}+\frac{\alpha}{k+1}-\frac{L h^2}{2}\right)\| \nabla f(x_{k+1}) \|^2  .
\end{align}
Recalling that $y_{k+1}=x_k-h^2\nabla f(x_k)$ and $x_{k+1}=x_k+hv_k$, and using the strong convexity of $f$, we deduce that
\begin{align} \label{E:G_SC_2}
f(x_{k+1}) - f(y_{k+1}) & \le \langle \nabla f(x_{k+1}),x_{k+1}-y_{k+1}\rangle-\frac{\mu}{2}\|x_{k+1}-y_{k+1}\|^2  \nonumber \\
& = h\langle \nabla f(x_{k+1}), v_k + h\nabla f(x_k) \rangle - \frac{\mu h^2}{2}\| v_k + h\nabla f(x_k) \|^2.
\end{align}
From \eqref{E: iterates}, we have
$$ v_k + h \nabla f(x_k) = \frac{k+\alpha+1}{k+1} v_{k+1} + \left( \gamma + 1 + \frac{\alpha}{k+1} \right) h\nabla f(x_{k+1}).$$
Combining this with \eqref{E:G_SC_2}, we obtain
\begin{align*} %\label{E:G_SC_3}
f(x_{k+1}) - f(y_{k+1})  
&\le  h\left(\frac{k+\alpha+1}{k+1}\right)\langle \nabla f(x_{k+1}), v_{k+1} \rangle - \frac{\mu h^2}{2}\| v_k + h\nabla f(x_k) \|^2\nonumber \\
&\quad +h^2\left( \gamma + 1 + \frac{\alpha}{k+1} \right)\|\nabla f(x_{k+1})\|^2 .
\end{align*}
Substituting in \eqref{E:G_SC_1}, we deduce that
\begin{equation} \label{E:G_SC_4}
E_{k+1} - E_k \le   - \frac{\mu\gamma h^2}{2}\| v_k + h\nabla f(x_k) \|^2   -\gamma h^2 \left(1-\frac{\gamma}{2}-\frac{L h^2}{2}\right)\| \nabla f(x_{k+1}) \|^2.
\end{equation}
The strong convexity of $f$ implies that $f(x_{k+1}) - f^* \le \frac{1}{2\mu}\| \nabla f(x_{k+1}) \|^2$. Using this, together with the $L$-smoothness of $f$, we obtain
\begin{align*}
f(y_{k+1}) - f^*
&= f(y_{k+1}) - f(x_{k+1}) + f(x_{k+1}) - f^* \\
&\le \langle \nabla f(x_{k+1}), y_{k+1} - x_{k+1} \rangle + \frac{L}{2}\| y_{k+1} - x_{k+1} \|^2 + \frac{1}{2\mu}\| \nabla f(x_{k+1}) \|^2.
\end{align*}
For every $\zeta,\xi\in H$ and $m >0$, we have $\langle \zeta,\xi\rangle \le \frac{1}{2m}\| \zeta \|^2 + \frac{m}{2}\|\xi \|^2$. Therefore, 
$$f(y_{k+1}) - f^* \le \left( \frac{1}{2m} + \frac{1}{2\mu} \right)\| \nabla f(x_{k+1}) \|^2 +  \frac{h^2(L+m)}{2} \| v_k + h\nabla f(x_k) \|^2,$$
for every $m>0$. Since $E_k = \frac{1}{2}\| v_k + h\nabla f(x_k) \|^2 + \gamma\left( f(y_{k+1}) - f^* \right)$, we obtain
\begin{equation}\label{E: E_bound_SC}
E_k\le \frac{1 + \gamma h^2(L+m)}{2} \| v_k + h\nabla f(x_k) \|^2 
     + \frac{\gamma}{2} \left( \frac{1}{m} + \frac{1}{\mu} \right)\| \nabla f(x_{k+1}) \|^2.
\end{equation}
Comparing \eqref{E:G_SC_4} and \eqref{E: E_bound_SC}, we deduce that
$$E_{k+1} - E_k 
\le -\rho E_k,   
$$
where 
$$\rho:= \max_{m>0}\min\left\{ \frac{\gamma\mu h^2}{1 + (L+m)\gamma h^2},  \frac{(2-\gamma - Lh^2) h^2}{ 1/m + 1/\mu } \right\}.$$
Bounding with $m=L$, we obtain the lower bound for $\rho$.\end{proof}

\begin{corollary}
Let $f$ be $\mu$-strongly convex and $L$-smooth. Let $(x_k)$ and $(y_k)$ be generated by
$$\left\{\begin{array}{ccl}
y_{k+1} &=& x_k - \frac{1}{2L}\nabla f(x_k) \\
x_{k+1} &=& y_{k+1} + \frac{k}{k+\alpha}\left( y_{k+1} - y_k \right),
\end{array}
\right.$$
where $\alpha>0$. Given $x_1$, set $v_1 = -\frac{1}{\sqrt{2L}}\nabla f(x_1)$. For every $k\ge 1$, we have
$$ f(y_{k+1}) - f^* \le \left( 1 - \frac{\mu}{4(L+\mu)} \right)^{k-1} \left( f(x_1)-f^* \right).$$
\end{corollary}

\begin{proof}
We use Theorem \ref{T:grad_SC} with $\gamma=1$ and $h^2 = \frac{1}{2L}$, to obtain
$$E_{k+1}\le(1-\rho)E_k,\qqbox{where}\rho \ge \frac{\mu}{2L}\min\left\{ \frac{1}{2},  \frac{1}{2( 1 + \frac{\mu}{L}) }  \right\}=\frac{\mu}{4(L+\mu)} .$$
Since $f(y_{k+1}) - f^* \le E_k$ for all $k\ge 1$, and $E_1=f(y_2)-f^*\le f(x_1) - f^*$, the result follows.
\end{proof}

\begin{remark}
A faster convergence rate of approximately $\mathcal{O}\left( \frac{1}{k^2}\left( 1 - \frac{\mu}{4L} \right)^k \right)$ has been established in \cite{Shi_2024,Bao_2023} for $\alpha\ge 3$.  
\end{remark}

%%%%%%%%%%%%%%%%%%%%%%%%%%%%%%%%%%%%%%%%
%%%%%%%%%%%%%%%%%%%%%%%%%%%%%%%%%%%%%%%%
\section{Convergence rates III: proximal algorithms}\label{Sec: proximal_algorithms}
In this section, we consider the convex composite optimization problem:
\begin{equation}\label{Problem: P_prox}
\min_{x\in H} F(x) = f(x) + g(x),
\end{equation}
where $f:H\to\mathbb{R}$ is convex and $L$-smooth; $g:H\to\mathbb{R}\cup\{\infty\}$ is convex, proper and lower semicontinuous. We assume that the solution set of problem \eqref{Problem: P_prox} is nonempty, and denote $F^* = \min(F)$.

To solve problem \eqref{Problem: P_prox}, one can employ the accelerated proximal-gradient method, also known as FISTA \cite{Beck_2009}, namely:
\begin{equation}\label{Algo: APM}\tag{APM}
\left\{
\begin{array}{rcl}
y_{k+1} &=& x_k - sG_s(x_k) \\
x_{k+1} &=& y_{k+1} + \frac{k}{k+\alpha}( y_{k+1} - y_k ),
\end{array}
\right.
\end{equation}
where $0< s \le \frac{1}{L}$ and the operator $G_s$ is given by
$$G_s(x): = \tfrac{x-\prox_{sg}(x-s\nabla f(x))}{s},\ \text{where, } \prox_{sg}(z) := \argmin_{u}\left\{ g(u) + \frac{1}{2s}\| u - z \|^2  \right\}.  $$
It follows from the definitions above that 
\begin{equation}\label{E: iterates_prox}
x_{k+1} - y_{k+1} = \frac{k}{k+\alpha}\big(x_k - y_k - sG_s(x_k)\big).
\end{equation}

The following auxiliary results will be useful in what follows. The first one is \cite[Lemma 7]{Bao_2023}, namely:

\begin{lemma}\label{Lem: prox_bound_SC}
Let $f:H\to\mathbb{R}$ be $\mu$-strongly convex and $L$-smooth, let $g:H\to\mathbb{R}\cup\{\infty\}$ be proper, convex and lower-semicontinuous, and set $F=f+g$. For every $x,y\in H$, we have
$$F\left(x-sG_s(x)\right) \le F(y) + \left\langle G_s(x), x-y \right\rangle - \frac{s(2-sL)}{2} \| G_s(x) \|^2 - \frac{\mu}{2}\| x-y \|^2.$$
If $f$ is merely convex, the inequality holds with $\mu=0$.
\end{lemma}

\begin{lemma} \label{L:PL_prox}
Let $f:H\to\mathbb{R}$ be convex and $L$-smooth, and let $g:H\to\mathbb{R}\cup\{\infty\}$ be proper, convex and lower-semicontinuous. Let the function $F=f+g$ satisfy the Polyak-\L ojasiewicz inequality with constant $\mu>0$. For every $x\in H$, we have
$$ \| G_s(x) \|^2 \ge 2\mu \left( F\left(x-sG_s(x)\right) - F^* \right). $$ 
\end{lemma}

\begin{proof}
Writing $y:=x-sG_s(x)$, we obtain $y = \prox_{sg}(x-s\nabla f(x))$, and then
$$ G_s(x) \in \nabla f(x) + \partial g(y). $$
Denote 
$$\xi := G_s(x) + \nabla f(y) - \nabla f(x)\in\partial F(y).$$
Since $f$ is convex and $L$-smooth, we can deduce
\begin{align*}
\| \xi \|^2
&= \| G_s(x) \|^2 + \| \nabla f(y) - \nabla f(x) \|^2 + 2 \langle G_s(x), \nabla f(y) - \nabla f(x)\rangle \\
&= \| G_s(x) \|^2 + \| \nabla f(y) - \nabla f(x) \|^2 - \frac{2}{s} \langle y-x, \nabla f(y) - \nabla f(x)\rangle \\
&\le \| G_s(x) \|^2 + \frac{sL-2}{sL}\| \nabla f(y) - \nabla f(x) \|^2 \\
&\le \| G_s(x) \|^2,
\end{align*}
because $sL\le 1$. By the Polyak-\L ojasiewicz inequality, we conclude that
$$  \| G_s(x) \|^2 \ge\| \xi \|^2 \ge 2\mu \left( F(y) - F^* \right), $$
as desired.
\end{proof}

%%%%%%%%%%%%%%%%%%%%%%%%%%%%%%%%%%%%%%%%
%%%%%%%%%%%%%%%%%%%%%%%%%%%%%%%%%%%%%%%%
\subsection{General estimations}
In this subsection, we present some useful estimations before we proceed to prove the convergence rate of the algorithm \eqref{Algo: APM}. Our convergence proofs center around the following energy-like sequence $(E_k)_{k\ge 1}$, which is defined by
\begin{equation}\label{E: E_k_prox}
E_k = \frac{1}{2}\| \phi_k \|^2 + s\left( F(y_k) - F^* \right),
\end{equation} 
where $\phi_k:= x_k - y_k$. We have the following:

\begin{proposition}\label{P: E_k_bound_prox}
Let $f$ be convex and $L$-smooth. Let $0 < s \le \frac{1}{L}$. Let $(x_k)_{k\ge 1}$ and $(y_k)_{k\ge 1}$ be generated according to \eqref{Algo: APM}, and consider the sequence $(E_k)_{k\ge 1}$ defined by \eqref{E: E_k_prox}. Then, 
\begin{align*}
E_{k+1} - E_k
    =& - \frac{\alpha}{k}\left(1+\frac{\alpha}{2k}\right)\| x_{k+1} - x_{k+1} \|^2
   - s\langle G_s(x_k), x_k - y_k \rangle  \\
   &\ + \frac{s^2}{2} \| G_s(x_k) \|^2 + s\left( F(y_{k+1}) - F(y_k) \right).
\end{align*}
\end{proposition}

\begin{proof}
By definition,
$$E_{k+1} - E_k = \left( \frac{1}{2}\| \phi_{k+1} \|^2 - \frac{1}{2}\| \phi_k \|^2 \right) + s\left( F(y_{k+1}) - F(y_k) \right).  $$
From \eqref{E: iterates_prox}, we obtain
$$ \phi_{k+1} - \phi_k = - \frac{\alpha}{k+\alpha}(x_k-y_k) - \frac{k s}{k+\alpha} G_s(x_k), $$
so that
\begin{align*}
\| \phi_{k+1} - \phi_k \|^2 
&= \left( \frac{\alpha}{k+\alpha} \right)^2 \| x_k - y_k \|^2 
  + \left( \frac{k}{k+\alpha} \right)^2 s^2 \| G_s(x_k) \|^2 \\
&\quad + \frac{2\alpha k s}{(k+\alpha)^2} \langle G_s(x_k), x_k-y_k \rangle, 
\end{align*}
and
\begin{align*}
\langle \phi_k, \phi_{k+1} - \phi_k \rangle
&= \left\langle x_k - y_k, - \frac{\alpha}{k+\alpha}(x_k-y_k) - \frac{k s}{k+\alpha} G_s(x_k) \right\rangle \\
&= - \frac{\alpha}{k+\alpha} \| x_k - y_k \|^2 - \frac{k s}{k+\alpha} \langle G_s(x_k), x_k - y_k \rangle.  
\end{align*}
Since 
$$\frac{1}{2}\| \phi_{k+1} \|^2 - \frac{1}{2}\| \phi_k \|^2 = \langle \phi_k, \phi_{k+1} - \phi_k \rangle + \frac{1}{2}\| \phi_{k+1} - \phi_k \|^2,$$ 
it follows that
\begin{align*}
\frac{1}{2}\| \phi_{k+1} \|^2 - \frac{1}{2}\| \phi_k \|^2  
= & - \frac{\alpha (2k+\alpha)}{2(k+\alpha)^2} \| x_k - y_k \|^2 
- \frac{k^2 s}{(k+\alpha)^2} \langle G_s(x_k), x_k - y_k \rangle  \\
& \  + \frac{s^2}{2}\left( \frac{k}{k+\alpha} \right)^2  \| G_s(x_k) \|^2   \\
= & \ - \frac{\alpha (2k+\alpha)}{2(k+\alpha)^2} \| x_k - y_k - s G_s(x_k) \|^2
   - s\langle G_s(x_k), x_k - y_k \rangle   \\
& \  + \frac{s^2}{2} \| G_s(x_k) \|^2.
\end{align*}
Using \eqref{E: iterates_prox}, we have
$$\frac{1}{2}\| \phi_{k+1} \|^2 - \frac{1}{2}\| \phi_k \|^2 = - \frac{\alpha (2k+\alpha)}{2k^2} \| x_{k+1} - x_{k+1} \|^2  - s\langle G_s(x_k), x_k - y_k \rangle 
   + \frac{s^2}{2} \| G_s(x_k) \|^2.$$
We conclude by adding $s\left( F(y_{k+1}) - F(y_k) \right)$ on both sides.
\end{proof}

%%%%%%%%%%%%%%%%%%%%%%%%%%%%%%%%%%%%%%%%
%%%%%%%%%%%%%%%%%%%%%%%%%%%%%%%%%%%%%%%%
\subsection{Convexity and Polyak-\L ojasiewicz inequality}
We begin by establishing a convergence rate for the function values in case the objective function is convex and satisfies the Polyak-\L ojasiewicz inequality. 

\begin{theorem}
Let $F=f+g$, where $f:H\to\mathbb{R}$ is convex and $L$-smooth, and $g:H\to\mathbb{R}\cup\{\infty\}$ is proper, convex and lower-semicontinuous. Assume $F$ satisfies the Polyak-\L ojasiewicz inequality with constant $\mu>0$. Let $(x_k)_{k\ge 1}$ and $(y_k)_{k\ge 1}$ be generated according to \eqref{Algo: APM}, with
$\alpha>0$, $0 < s \le \frac{1}{L}$ and $x_1 = y_1$. For every $k\ge 1$, we have
\begin{equation*}
F(y_{k+1}) - F^*\le\left\{
\begin{array}{ll}
\frac{ F(y_1) - F^* }{ \left[ 1 + \mu s (1-sL) \right]^{k-1} },&\quad 1\le k \le K,\\ [5pt]
\frac{F(y_1) - F^*}{\left[ 1 + \mu s (1-sL) \right]^{K-1} } \left( \frac{K+\alpha}{k+\alpha} \right)^{2\alpha},&\quad k> K,
\end{array}
\right.
\end{equation*}
where $K$ is the largest integer such that $\mu s (1-sL)(K-1)^2-2\alpha(K-1)-\alpha^2\le 0$. Moreover, $F(y_{k+1}) - F^* \le o( \frac{1}{k^{2\alpha}})$ as $k\to\infty$.
\end{theorem}

\begin{proof}
Setting $x=x_k$ and $y=y_k$ in Lemma \ref{Lem: prox_bound_SC} gives
$$F(y_{k+1}) - F(y_k) \le \langle G_s(x_k), x_k - y_k \rangle - s\left( 1 - \frac{1}{2}sL \right) \| G_s(x_k) \|^2.$$
Combining this with Proposition \ref{P: E_k_bound_prox}, and then using Lemma \ref{L:PL_prox}, we obtain
\begin{align}
E_{k+1} - E_k
\le & - \frac{\alpha}{k}\left(1+\frac{\alpha}{2k}\right)\| x_{k+1} - x_{k+1} \|^2 - \frac{s^2(1-sL)}{2} \| G_s(x_k) \|^2 \nonumber \\
\le & - \frac{\alpha}{k}\left(1+\frac{\alpha}{2k}\right)\| x_{k+1} - x_{k+1} \|^2 - \mu s^2(1-sL)\big(F(y_{k+1})-F^*\big) \label{E:O_to_o_prox}\\
\le & -\min\left\{ \frac{2\alpha}{k}\left(1+\frac{\alpha}{2k}\right),\mu s(1-sL)\right\}E_{k+1},\nonumber 
\end{align}
by the definition of $(E_k)_{k\ge 1}$. We now proceed as in the proof of Theorem \ref{Thm: algo_PL}. If $k\le K$, then
$$E_k\le\frac{E_1}{\big(1+\mu s(1-sL)\big)^{k-1}}.$$
For $k>K$, we have
$\left(1+\frac{\alpha}{k}\right)^2E_{k+1}\le E_k$,
and so
$$E_k\le \left(\frac{K+\alpha}{k+\alpha}\right)^{2\alpha}E_K.$$
Since $F(y_k)-F^*\le \frac{E_k}{s}$ and $E_1=s\big(F(y_1)-F^*\big)$, the conclusion easily follows. For all $k$ large enough, we have $\frac{2\alpha}{k}\left(1+\frac{\alpha}{2k}\right)\le\frac{1}{2}\mu s(1-sL)$. We can use \eqref{E:O_to_o_prox} to deduce that
$$\left(1+\frac{\alpha}{k}\right)^2E_{k+1}- E_k + \frac{\mu s(1-sL)}{2}\big(F(y_{k+1})-F^*\big)\le 0.$$
Multiply this by $(k+\alpha)^{2\alpha}$ to obtain
$$(k+\alpha)^{2\alpha}\left(1+\frac{\alpha}{k}\right)^2E_{k+1}- (k+\alpha)^{2\alpha}E_k + \delta_k\le 0,$$
where $\delta_k:= \frac{\mu s(1-sL)}{2}(k+\alpha)^{2\alpha}\big(F(y_{k+1})-F^*\big)$. As in the proof of Theorem \ref{Thm: algo_PL}, we first show that $(k+\alpha)^{2\alpha}\left(1+\frac{\alpha}{k}\right)^2\ge (k+1+\alpha)^{2\alpha}$, then deduce that $\sum_{k=1}^\infty\delta_k<\infty$, and finally conclude that $\lim_{k\to\infty}\delta_k=0$.
\end{proof}

%%%%%%%%%%%%%%%%%%%%%%%%%%%%%%%%%%%%%%%%
%%%%%%%%%%%%%%%%%%%%%%%%%%%%%%%%%%%%%%%%
\subsection{Strong convexity}
In this subsection, we prove the linear convergence rate for the algorithm \eqref{Algo: APM} in case $f$ is strongly convex.

\begin{theorem}
Let $F=f+g$, where $f:H\to\mathbb{R}$ is $\mu$-strongly convex and $L$-smooth, and $g:H\to\mathbb{R}\cup\{\infty\}$ is proper, convex and lower-semicontinuous. Generate $(x_k)_{k\ge 1}$ and $(y_k)_{k\ge 1}$ following \eqref{Algo: APM}, with
$\alpha>0$, $0 < s < \frac{1}{L}$ and $x_1 = y_1$. For every $k\ge 1$, we have
$$ F(y_k) - F^* \le \frac{F(x_1) - F^*}{(1+\rho)^{k-1}}, $$
where $\rho = \mu \min\left\{\frac{s}{2}, \frac{s(1-sL)}{1+\mu Ls^2} \right\}$.
\end{theorem}

\begin{proof}
The proof is similar to that of Theorem \ref{T:grad_SC}. On the one hand, we use \eqref{E: iterates_prox} to show that
\begin{align}\label{E: phi_k_bound_prox}
\frac{1}{2}\| x_{k+1} - y_{k+1} \|^2
&\le \left( \frac{k}{k+\alpha} \right)^2 \left[ \frac{1 + m}{2} \| x_k - y_k \|^2
+ \frac{s^2}{2}\left( 1 + \frac{1}{m} \right) \| G_s(x_k) \|^2 \right] \nonumber \\
&\le \frac{1 + m}{2} \| x_k - y_k \|^2
+ \frac{s^2}{2}\left( 1 + \frac{1}{m} \right) \| G_s(x_k) \|^2.
\end{align} 
On the other hand, setting $x=x_k$ and $y=x^*$ in Lemma \ref{Lem: prox_bound_SC} gives
\begin{align*} %\label{E: F_diff_bound_prox} 
F(y_{k+1}) - F^* 
&\le \langle G_s(x_k), x_k - x^* \rangle - \frac{s(2-sL)}{2} \| G_s(x_k) \|^2 - \frac{\mu}{2}\| x_k - x^* \|^2 \\    
& \le \frac{1}{2\mu}\| G_s(x_k) \|^2+\frac{\mu}{2}\| x_k - x^* \|^2 - \frac{s(2-sL)}{2} \| G_s(x_k) \|^2 - \frac{\mu}{2}\| x_k - x^* \|^2 \\    
& = \left[ \frac{1}{2\mu} - \frac{s(2-sL)}{2}\right] \| G_s(x_k) \|^2.
\end{align*}
Combining this with \eqref{E: phi_k_bound_prox}, we obtain
\begin{equation} \label{E:SC_prox_aux1}
E_{k+1}\le \frac{1 + m}{2} \| x_k - y_k \|^2
+\frac{s}{2}\left[ \frac{1}{\mu} +s\left(sL - 1 +\frac{1}{m}\right) \right] \| G_s(x_k) \|^2.
\end{equation}
Setting $x=x_k$ and $y=y_k$ in Lemma \ref{Lem: prox_bound_SC} gives
$$F(y_{k+1}) - F(y_k) \le \langle G_s(x_k), x_k - y_k \rangle - \frac{s(1-sL)}{2} \| G_s(x_k) \|^2-\frac{\mu}{2}\|x_k-y_k\|^2.$$
Combining this with Proposition \ref{P: E_k_bound_prox},
we get 
\begin{equation} \label{E:SC_prox_aux2}
E_{k+1} - E_k
\le  - \frac{\mu s}{2}\|x_k-y_k\|^2   - \frac{s^2(1-sL)}{2} \| G_s(x_k) \|^2.
\end{equation}
By comparing \eqref{E:SC_prox_aux1} and \eqref{E:SC_prox_aux2}, we deduce that
$$(1+\rho)E_{k+1} \le E_k,$$
where
$$\rho=\max_{m>0}\min\left\{\frac{\mu s}{1+m}, \frac{\mu s(1-sL)}{ 1 +\mu s\left( sL - 1 + \frac{1}{m} \right) }\right\},$$
and the lower bound is obtained by choosing $m=1$.
\end{proof}

\begin{remark}
For any $\alpha>0$, the step size $s=\frac{1}{2L}$ gives $F(y_k)-F^* \le \mathcal{O}\left( \frac{1}{( 1 + \rho)^k}  \right)$, with $\rho = \frac{\mu}{4L+\mu}$. A faster convergence rate of approximately $\mathcal{O}\left( \frac{1}{k^2}\left( 1 - \frac{\mu}{4L} \right)^k \right)$ was obtained in \cite{Shi_2024,Bao_2023} for $\alpha\ge 3$.   
\end{remark}

%%%%%%%%%%%%%%%%%%%%%%%%%%%%%%%%%%%%%%%%
%%%%%%%%%%%%%%%%%%%%%%%%%%%%%%%%%%%%%%%%
\bibliographystyle{siamplain}
\bibliography{myrefs}

\end{document}